\documentclass[12pt,a4paper]{amsart}
\usepackage[latin1]{inputenc}
\usepackage{latexsym}
\usepackage{color,graphicx,shortvrb}
\usepackage{amsmath, amssymb}
\usepackage{amsfonts}
\usepackage[colorlinks, bookmarks=true]{hyperref}

\newtheorem{theorem}{Theorem}[section]
\newtheorem{lemma}[theorem]{Lemma}
\newtheorem{proposition}[theorem]{Proposition}
\newtheorem{corollary}[theorem]{Corollary}

\theoremstyle{definition}
\newtheorem{definition}[theorem]{Definition}
\newtheorem{remark}[theorem]{Remark}

\newtheorem{example}[theorem]{Example}

\newcommand{\N}{\mathbb{N}}

\title{On Shapiro's Lethargy Theorem and some applications}
\author{ A. G. Aksoy, J. M. Almira}

\begin{document}


\baselineskip=16pt

\numberwithin{equation}{section}

\maketitle \markboth{On Shapiro's Lethargy Theorem}{A. G. Aksoy, J. M. Almira}

\begin{abstract}
Shapiro's lethargy theorem \cite{shapiro} states that if $\{A_n\}$ is any non-trivial linear approximation scheme on a Banach space $X$, then the sequences of errors of best 
approximation $E(x,A_n) = \inf_{a \in A_n} \|x - a_n\|_X$ decay almost arbitrarily slowly. 
Recently, Almira and Oikhberg \cite{almira_oikhberg}, \cite{almira_oikhberg_2} investigated this kind of result for general approximation schemes in the quasi-Banach setting. 
In this paper, we consider the same question for $F$-spaces with non decreasing
metric $d$. We also provide  applications to the rate of decay of s-numbers,
entropy numbers, and slow convergence  of sequences of operators.
\end{abstract}

\section{Motivation}
A famous theorem by Kakutani \cite{kakutani} states that a topological vector space is metrizable if and only if it contains a countable basis of neighborhoods. Furthermore, if the topological vector space $X$ admits a compatible metric $d$ then it also admits an equivalent metric $d^*$ which is translation invariant. Thus, we assume in all what follows that our metrics are translation invariant. Recall that a metric vector space $(X,d)$ is named an $F$-space if and only if it is complete. There are two large categories in these spaces: the locally bounded and the locally convex ones. An important result by Aoki \cite{aokiresult} and Rolewicz \cite{rolewiczresult} guarantees that every locally bounded metric vector space admits a compatible $p$-norm, so that the class of locally bounded $F$-spaces coincides with the class of quasi-Banach spaces. On the other hand, it is also well known that normable metric spaces are precisely the metric vector spaces which are  locally bounded and locally convex (this result was proved by Kolmogorov in 1935 \cite{kolmorogov}).  In particular, an $F$-space is Banach if and only if it is locally bounded and locally convex.

Given $(X,d)$  an $F$-space, and
$A_0\subset A_1\subset\ldots \subset A_n\subset\ldots \subset X$
an infinite chain of subsets of $X$, where all inclusions are
strict,  we say that $(X,\{A_n\})$ is an {\it approximation scheme} (or that $(A_n)$ is an approximation scheme in $X$) if:
\begin{itemize}
\item[$(A1)$] There exists a map $K:\mathbb{N}\to\mathbb{N}$ such that $K(n)\geq n$ and $A_n+A_n\subseteq A_{K(n)}$ for all $n\in\mathbb{N}$,

\item[$(A2)$] $\lambda A_n\subset A_n$ for all $n\in\mathbb{N}$ and all scalars $\lambda$,

\item[$(A3)$] $\bigcup_{n\in\mathbb{N}}A_n$ is a dense subset of $X$.
\end{itemize}

Approximation schemes were introduced in Banach space theory by Butzer and Scherer in 1968 \cite{butzer_scherer} and, independently,  by Y. Brudnyi and N. Kruglyak under the name of ``approximation families'' in 1978 \cite{brukru}. They were
popularized by Pietsch in his seminal paper of 1981 \cite{Pie}, which studied the approximation spaces
$A_p^r(X,{A_n})=\{x\in X: \|x\|_{A_p^r}=\|\{E(x,A_n)\}_{n=0}^\infty\|_{\ell_{p,r}}<\infty\}$. Here, $$\ell_{p,r}=\{\{a_n\}\in\ell_\infty: \|\{a_n\}\|_{p,r}=\left[\sum_{n=1}^\infty n^{rp-1}(a_n^*)^p\right]^\frac{1}{p}<\infty\}$$ denotes the so called Lorentz sequence space, $(X,\|\cdot\|_X)$ is a quasi-Banach space and $E(x,A_n)=\inf_{a\in A_n}\|x-a\|_X$.

A fundamental part of the  theory developed by the authors of the above mentioned papers consists of the study of the embeddings between the involved spaces. In particular, Pietsch proved that the embedding $A_p^r(X,{A_n})\hookrightarrow A_q^s(X,{A_n})$ holds true whenever $r>s>0$ or $r=s$ and $p<q$. This, in conjunction with the central theorems in approximation theory, which state a strong relation between smoothness of functions $f$ (compactness of operators $T$, respectively) and  fast decay of approximation errors $E(f,A_n)$ (approximation numbers $a_n(T)$, respectively), has been used to speak about the scale of smoothness (compactness, respectively) defined by an approximation scheme $(X,\{A_n\})$. Concretely, it is assumed (see, for example,  \cite{almirabjma}, \cite{devorenonlinear}, \cite{devore}, \cite{PieID}) that membership to the approximation space $A_p^r(X,\{A_n\})$ is a concept of smoothness (compactness if $X=B(Y_1,Y_2)$ and $A_n=\{T\in B(Y_1,Y_2): \textbf{rank} (T)<n\}$).
Approximation schemes are, thus,
a natural subject of study in Approximation Theory.
Indeed, the approximation scheme concept is the abstract tool that models
all approximation processes, and can be considered as
a central concept for the theory.

Another main motivation for Pietsch's contribution \cite{Pie} was the existence of a strong  parallelism between the theories of approximation spaces and interpolation spaces. In particular, he proved embedding, reiteration and representation results for his approximation spaces. Simultaneously and also independently, Ti\c{t}a \cite{tita0} studied, from 1971 on, for the case of approximation of linear operators by finite rank operators,  a similar concept, based on the use of symmetric norming functions $\Phi$ 
and the sequence spaces defined by them, $S_{\Phi}=\{\{a_n\}:\exists \lim_{n\to\infty}\Phi(a_1^*,a_2^*,\cdots,a_n^*,0,0,\cdots)\}$.
The concept of approximation scheme given in the present paper generalizes to the $F$-spaces setting  a definition which was introduced by Almira and Luther \cite{almiraluther2}, \cite{almiraluther1} some time  ago. They also created a  theory for generalized approximation spaces via the use of general sequence spaces $S$ (that they named ``admissible sequence spaces'') and the definition of the approximation spaces $A(X,S,\{A_n\})=\{x\in X: \|x\|_{A(X,S)}=\|\{E(x,A_n)\}\|_S<\infty\}$. Other papers with a similar spirit of generality have been written by Aksoy \cite{Ak0,Ak1, Ak2, AkNa}, Ti\c{t}a \cite{tita_cluj_99} and Pustylnik \cite{pustylnik}, \cite{pustylnik2}. Finally, a few other important references for people interested on approximation spaces are \cite{cobos}, \cite{cobos_milman}, \cite{cobos_resina}, \cite{feher_g}, \cite{peetre_sparr}, \cite{tita_anal}, \cite{tita_col, tita_col2} and \cite{tita_studia}. It is important to remark that, due to the centrality of the concept of approximation scheme in approximation theory,  the idea of defining approximation spaces is a quite natural one. Unfortunately, this has had the negative effect that many unrelated people has thought on the same things at different places and different times, and some papers in this subject partially overlap.

In this paper, we study the behavior of best approximation errors
of an element $x \in X$ relative to an approximation scheme when $(X,d)$ is an $F$-space.

To proceed further, we establish our notation. We write $\{\varepsilon_i\} \searrow 0$
to indicate the sequence $\varepsilon_1 \geq \varepsilon_2 \geq \ldots \geq 0$
satisfies $\lim_i \varepsilon_i = 0$. For an $F$-space $(X,d)$, we denote by $B_d(x,r)$
and $S_d(x,r)$ the closed ball and the sphere of center  $x\in X$ and radius $r>0$, respectively. That is,
$S_d(x,r) = \{y \in X : d(x,y) = r\}$, and $B_d(x,r) = \{y \in X : d(x,y) \leq r\}$.  We use the notation $B(x,r)$ and $S(x,r)$ if there is no possibility of confusion with respect to the metric $d$ we are dealing with.
If  $x \in X$, and $A \subset X$, we define
the {\it best approximation error} of $x$ with respect to $A$ by
$E(x,A)_X = \inf_{a \in A} d(x,a)$.
When there is no confusion as to the ambient space $X$,
we simply use the notation $E(x,A)$. If $B$ and $A$ are two subsets of $X$ and $\lambda$ is an scalar,
we set $A+B=\{x+y:x\in A \text{ and } y\in B\}$, $\lambda A=\{\lambda x:x\in A\}$, and $E(B,A) = \sup_{b \in B} E(b,A)$. Note that $E(B,A)$ may be different
from $E(A,B)$.
Finally, we recall that  $A\subset X$ is bounded if for every $r>0$ there exists $\lambda>0$ such that $A\subseteq \lambda B(0,r)$. This is quite different of being $d$-bounded, which means that $A\subseteq B(0,r)$ for a certain $r>0$.

The results described below have their origins in  the classical Lethargy Theorem by
S.N. Bernstein \cite{bernsteininverso}, stating that, for any linear
approximation scheme $\{A_n\}$ in a Banach space $X$, if $\dim A_n<\infty$ for all $n$ and $\{\varepsilon_n\}\searrow 0$, there exists
$x \in X$ such that $E(x,A_n)=\varepsilon_n$ for all $n\in \N$.
Bernstein's proof is based on a compactness argument, where he imposed  $\dim A_n<\infty$ for all $n$. In 1964 H.S. Shapiro \cite{shapiro} used Baire's category theorem and Riesz's lemma (on the existence of almost orthogonal elements to any closed linear subspace $Y$ of a Banach space $X$) to prove that,
for any sequence $A_1 \subsetneq A_2 \subsetneq \ldots \subsetneq X$ of
closed but not necessarily finite dimensional subspaces of a Banach space $X$,
and any sequence $\{\varepsilon_{n}\}\searrow 0$, there exists an $x\in X$ such that $E(x,A_{n})\neq\mathbf{O}(\varepsilon_{n})$.  This result was strengthened by Tjuriemskih \cite{tjuriemskih1}, who, under the very same conditions of Shapiro's Theorem, proved the existence of  $x\in X$ such that $E(x,A_{n})\geq \varepsilon_{n}$, $n=0,1,2,\cdots$. Moreover, Borodin \cite{borodin} gave a new easy proof of this result and proved that, for arbitrary infinite dimensional Banach spaces $X$ and for sequences $\{\varepsilon_n\}\searrow 0$ satisfying $\varepsilon_n>\sum_{k=n+1}^\infty\varepsilon_k$, $n=0,1,2,\cdots$, there exists  $x\in X$ such that $E(x,X_{n})= \varepsilon_{n}$, $n=0,1,2,\cdots$.

Motivated by these results, in \cite{almira_oikhberg} the authors gave several characterizations of
the approximation schemes with the property that for every non-increasing sequence
$\{\varepsilon_n\}\searrow 0$ there exists an element $x\in X$ such that
$E(x,A_n) \neq \mathbf{O}(\varepsilon_n)$. In this case we say that $\{A_n\}$ {\it satisfies Shapiro's Theorem on $X$}. In particular, Shapiro's original theorem claims that all non-trivial linear approximation schemes $(X,\{A_n\})$ with $X$ a Banach space, satisfy a result of this kind.

Let us introduce yet another definition: We say that the subset $Y$ of $X$ satisfies Shapiro's theorem with respect to the approximation scheme $(X, \{A_n\})$ if  for every sequence
$\{\varepsilon_n\}\searrow 0$ there exists an element $x\in Y$ such that
$E(x,A_n) \neq \mathbf{O}(\varepsilon_n)$. In order to simplify notation, and when there is no confusion, we will just say that the approximation scheme $\{A_n\}$ satisfies Shapiro's theorem on $Y$.

Now, while studying these problems for general approximation schemes, the following results were proved (see \cite[Theorem 2.2, Corollary 3.7]{almira_oikhberg}, \cite[Theorems  2.9, 4.2, 4.3 and 7.7]{almira_oikhberg_2}):
\begin{theorem}\label{teo_introd}
Let $X$ be a quasi-Banach space.  For any approximation scheme $(X,\{A_n\})$, the following are equivalent:
\begin{itemize}
\item[$(a)$]  The approximation scheme $\{A_n\}$ satisfies Shapiro's Theorem on $X$.

\item[$(b)$] There exists a constant $c>0$ and an infinite set $\mathbb{N}_{0}\subseteq\mathbb{N}$ such that for all $n\in\mathbb{N}_{0}$, there
exists some $x_{n}\in X\setminus \overline{A_{n}}$ which satisfies $E(x_{n},A_{n})\leq cE(x_{n},A_{K(n)}).$

\item[$(c)$]
There is no decreasing sequence $\{\varepsilon_n\}\searrow 0$ such that
$E(x,A_n)\leq \varepsilon_n\|x\|$  for all $x\in X$ and $n\in \N$.

\item[$(d)$] $E(S(0,1),A_n)=1$, $n=0,1,2,\ldots$.

\item[$(e)$] There exists $c>0$ such that $E(S(0,1),A_n)\geq c$, $n=0,1,2,\ldots$.
\end{itemize}
Moreover, if $X$ is a Banach space, then all these conditions are equivalent to:
\begin{itemize}
\item[$(f)$] For every non-decreasing sequence $\{\varepsilon_n\}_{n=0}^{\infty}\searrow 0$ there exists an element $x\in X$ such that $E(x,A_n)\geq \varepsilon_n$ for all $n\in\mathbb{N}$.
\end{itemize}
\end{theorem}

\begin{theorem}\label{teo_introd2} Let $X$ be a quasi-Banach space and assume that $(X,\{A_n\})$ is an approximation scheme which satisfies Shapiro's theorem on $X$.  If $Y\subseteq X$ is a finite codimensional subspace of $X$, then $\{A_n\}$ satisfies Shapiro's theorem on $Y$. If, furthermore, $X$ is Banach and $Y$ is closed in the topology of $X$, then for every sequence $\{\varepsilon_n\}\searrow 0$ there exists $y\in Y$ such that $E(y,A_n)_X\geq \varepsilon_n$ for all $n\in\mathbb{N}$.
\end{theorem}

\begin{theorem}\label{teo_introd3} Let $X$ be a quasi-Banach space and assume that $(X,\{A_n\})$ is an approximation scheme such that $A_n$ is boundedly compact on $X$ for all $n\in\mathbb{N}$.  If $Y\subseteq X$ is an  infinite dimensional closed subspace of $X$, then $\{A_n\}$ satisfies Shapiro's theorem on $Y$.
\end{theorem}

\begin{theorem}\label{teo_introd4} Let $X$ be a Banach space and assume that $(X,\{A_n\})$ is linear approximation scheme on $X$ such that $\dim A_n<\infty$ for all $n\in\mathbb{N}$. If $Y\subseteq X$ is an  infinite dimensional closed subspace of $X$, then for every sequence $\{\varepsilon_n\}\searrow 0$ there exists $y\in Y$ such that $\|y\|=\varepsilon_0$ and $E(y,A_n)_X\geq \varepsilon_n$ for all $n\geq 1$.
\end{theorem}

\begin{theorem}\label{teo_introd5}  Suppose  $Y\subseteq X$ is an  infinite dimensional closed subspace of a Banach space $X$, and $\textit{E}=\{e_i\}_{i=0}^\infty$ is an unconditional basis of $X$. Set
\[
\Sigma_n(\textit{E})=\bigcup_{I\subseteq \mathbb{N}, \#(I)=n}\textbf{span}\{e_i:i\in I\}, \ \ (n\in\mathbb{N}).
\]
Then, $Y$ satisfies Shapiro's theorem with respect to the approximation scheme $\{\Sigma_n(\textit{E})\}$.
\end{theorem}

The main goal of this paper is to initiate a study about approximation schemes that satisfy Shapiro's theorem in the $F$-spaces setting. This question was studied, for the case of linear approximation schemes, by G. Albinus \cite{albinus}. In section 2  we characterize, for a large class of $F$-spaces $(X,d)$, the approximation schemes $\{A_n\}$ which satisfy Shapiro's theorem on $X$   and, as a consequence, we give a new proof of Albinus's theorem and we show a few more examples of approximation schemes satisfying Shapiro's theorem on $F$-spaces. In section 3 we use the ideas of Section 2 to prove a general lethargy result for arbitrary scales of numbers and, as a consequence, we prove that, for a large class of quasi Banach spaces $X$, all s-number sequences $s_n(T)$ satisfy Shapiro's theorem in $\mathcal{L}(X)$.	 Finally, in section 4, we add some new applications of the lethargy results to the study of slow convergence of sequences of (possibly nonlinear) operators, a subject which has been recently investigated by Deutsch and Hundal for the case of continuous linear operators in the Banach setting \cite{DH1,DH2,DH3}.

\section{Shapiro's Theorem for $F$-spaces}\label{criteria}

Let us start with some general considerations about approximation schemes and the Shapiro's theorem. The first observation is that approximation schemes are only interesting in the infinite dimensional context:

\begin{proposition} \label{cero} Let $(X,d)$ be a metric vector space. If $\dim(X)<\infty$ and $\{A_n\}$ satisfies the conditions $(A1)$, $(A2)$ and $(A3)$ above, then there exists $N\in \mathbb{N}$ such that $A_N=X$.
\end{proposition}

\begin{proof} Let us set $X_n=\mathbf{span}(A_n)$, $n=0,1,\cdots$. Obviously, $X_n$ is a closed subspace of $X$ (since $s=\dim(X)<\infty$, which implies that all its subspaces are closed). Moreover, $\bigcup_nX_n$ is a dense subspace of $X$. Then Baire category theorem claims that there exists $m\in\mathbb{N}$ such that $X_m$ has non-empty interior. Assume that $B(x,r)\subset X_m$. Then $B(-x,r)\subset X_m$ and
\[
B(0,r)\subseteq \frac{1}{2}(B(x,r)+B(-x,r)) \subset X_m,
\]
since $d$ is translation invariant, which implies that, if $d(z,0)\leq r$ then $d(x+z,x)\leq r$, $d(-x+z,-x)\leq r$ and $z=\frac{1}{2}((x+z)+(-x+z))$. It follows that $X_m=X$ since the balls $B(0,r)$ are absorbing subsets of $X$. Now, $A_m$ spans $X_m$, so that we can take an algebraic basis of $X_m=X$ formed by elements of $A_m$. In particular, every  $x\in X$ is a finite sum $\sum_{k=1}^s \lambda_k a_k$ of elements of $A_m$ (since $\lambda A_m\subseteq A_m$ for all scalar $\lambda$). On the other hand, $A_m+A_m+\cdots A_m$ ($s$ times) is a subset of $A_h(m)$, where $h(m)=K(K(\cdots K(m)\cdots ))=N\in\mathbb{N}$ is a fixed finite number. This ends the proof.
\end{proof}

In the normed and quasi-normed setting, two (quasi-)norms $\|\cdot\|_1$ and $\|\cdot\|_2$, defined over the same vector space $X$, are equivalent (i.e., define the same topology on $X$) if and only if there exists two constants $C_1,C_2>0$ such that $C_1\|x\|_2\leq \|x\|_1\leq C_2\|x\|_2$ for all $x\in X$. This has a nice consequence that an approximation scheme $\{A_n\}$ in $X$ satisfies Shapiro's theorem with respect to the (quasi-) norm $\|\cdot\|_1$ if and only if it satisfies Shapiro's theorem with respect to any equivalent (quasi-)norm $\|\cdot\|_2$. In the case of $F$-spaces the question is much more delicate, since the equivalence of two distances $d_1$, $d_2$ is a much more subtle concept. Recall that metrics $d_1$ and $d_2$ over the vector space $X$ are equivalent  if they generate the same topology on $X$.

\begin{definition}[Rolewicz, \cite{rolewicz}] The metric $d$ is non-decreasing if $d(\alpha x, 0)\leq d(x,0)$ whenever $0\leq \alpha\leq 1$.
\end{definition}

\begin{remark}\begin{itemize}
\item[$(i)$] If the metric $d$ is non decreasing, then $d(\alpha x, 0)\leq d(\beta x,0)$ whenever $0\leq \alpha\leq \beta$ since, if  $0\leq \alpha< \beta$, then $0\leq \frac{\alpha}{\beta}\leq 1$ and $d(\alpha x,0)=d(\frac{\alpha}{\beta}\beta x,0) \leq d(\beta x,0)$.

\item[$(ii)$] Assume that $d$ is non decreasing and $0<\lambda\leq n$ with $n\in \mathbb{N}$. Then $d(\lambda x,0)=d(\frac{\lambda}{n} n x,0) \leq d(n x,0)\leq n d(x,0)$ for all $x\in X$. Consequently, if $A\subset X$ satisfies $\alpha A\subset A$ for all scalar $\alpha$, then
\[
E(\alpha x,A)\leq ([\alpha]+1)E(x,A) \text{ for all } x\in X \text{ and } \alpha>0.
\]
(Here $[\alpha]$ denotes the integral part of $\alpha$).

\item[$(iii)$] If $(X,d)$ is a metric vector space, then $d^*(x,y)=\sup_{0\leq t\leq 1}d(tx,ty)$ defines a non-decreasing equivalent metric on $X$ ( see \cite[Theorem 1.2.2.]{rolewicz}).
\end{itemize}
\end{remark}

\begin{example} If  $(X,d)$ is a locally convex metrizable topological vector space with metric $d(x,y)=\max_{k\in\mathbb{N}}\{2^{-k}\frac{p_k(x-y)}{1+p_k(x-y)}\}$, where $\{p_k\}$ is a separating family of semi norms of $X$, then $\|x\|_X=d(x,0)$ satisfies
\[
\min\{\alpha,1\}\|x\|_X\leq \|\alpha x\|_X\leq \max\{\alpha,1\}\|x\|_X \text{ for all } x\in X \text{ and } \alpha>0.
\]
In particular, $d$ is non-decreasing. Furthermore, if $A\subset X$ satisfies $\lambda A\subset A$ for all scalar $\lambda$, then
\[
\min\{\alpha,1\}E(x,A)\leq E(\alpha x,A)\leq \max\{\alpha,1\}E(x,A) \text{ for all } x\in X \text{ and } \alpha>0.
\]
\end{example}

\begin{proof} To prove this result it is enough to demonstrate that, if $t,\alpha\geq 0$,  then
\begin{equation}
\min\{\alpha,1\}\frac{t}{1+t}\leq \frac{\alpha t}{1+\alpha t} \leq \max\{\alpha,1\}\frac{t}{1+t}.
\end{equation}
These inequalities follow directly from the fact that $\phi(t)=\frac{t}{1+t}$ is an increasing function on $(0,+\infty)$ since, if $\alpha>1$, then 
\[
\frac{t}{1+t}\leq \frac{\alpha t}{1+\alpha t} \leq \alpha\frac{t}{1+t},
\]
and, if $\alpha<1$, then 
\[
\alpha \frac{t}{1+t}\leq \frac{\alpha t}{1+\alpha t} \leq \frac{t}{1+t},
\]
which is what we wanted to prove.
\end{proof}
\begin{example}[Musielak and Orlicz] Let  $X$ be a vector space with a metrizing modular $\rho(x)$, then
$X_{\rho}=\{x\in X:\rho(x)<\infty\}$ is a vector space and
\[
d_{\rho}(x,y)=\inf\{\varepsilon>0:\rho\left(\frac{x-y}{\varepsilon}\right)<\varepsilon\}
\]
defines a metric on $X_{\rho}$ which is non decreasing and invariant by translations . Moreover, if $\|x\|_{\rho}=d_{\rho}(x,x)$, then $\{\|x_n\|_{\rho}\}\to 0$ if and only if $\rho(x_n)\to 0$ (See \cite[page 6, Proposition 1.2.1. and Theorem 1.2.4.]{rolewicz} for the definition of modulars and a proof of this result).
 \end{example}

\begin{definition} Let $\{A_n\}$ be an approximation scheme on the $F$-space $(X,d)$, and let $B\subseteq X$. We say that:
\begin{itemize}
\item[$(a)$]$\{A_n\}$ fails Shapiro's theorem on $B$ if there exists $\{\varepsilon_n\}\searrow   0$ such that, if $x\in B$ then $E(x,A_n)\leq C(x)\varepsilon_n$ for all $n\in\mathbb{N}$ and a certain constant $C(x)>0$.
\item[$(b)$] $\{A_n\}$ fails Shapiro's theorem uniformly on $B$ if there exists $\{\varepsilon_n\}\searrow   0$  such that, $E(x,A_n)\leq  \varepsilon_n$ for all $n\in\mathbb{N}$ and all $x\in B$.
\item[$(c)$]  $\{A_n\}$ satisfies Shapiro's theorem on $B$ if for every  $\{\varepsilon_n\}\searrow   0$ there exists $x\in B$ such that $E(x,A_n)\neq \mathbf{O}(\varepsilon_n)$.
\end{itemize}
\end{definition}

\begin{theorem} \label{ST_fail} Let $(X,d)$ be an $F$-space and assume that $d$ is non-decreasing. Let $\{A_n\}$ be an approximation scheme in $X$. Then the following are equivalent claims:
\begin{itemize}
\item[$(i)$] $\{A_n\}$ fails Shapiro's theorem in $X$.
\item[$(ii)$] There exists $r_0>0$ such that  $\{A_n\}$ fails Shapiro's theorem uniformly on the ball $B(0,r_0)$.
\end{itemize}
Consequently, the approximation scheme $\{A_n\}$ satisfies Shapiro's theorem on $X$ if and only if  $\inf_{n\in\mathbb{N}}E(B(0,r),A_n)>0$ for all $r>0$.
\end{theorem}
For the proof of this result we need to use the following general property about sequences of positive real numbers:

\begin{lemma}\label{sucesiones} Given $\{\varepsilon_n\}\searrow   0$ and $\{h(n)\}$ an increasing sequence of natural numbers satisfying  $n\leq h(n)$ for all $n$, there exists a sequence
$\{\xi_n\}\searrow   0$ such that $\varepsilon_n\leq \xi_n$ and $\xi_n\leq 2\xi_{h(n)}$ for all $n$.
\end{lemma}

\begin{proof} See \cite[Lemma 2.3]{almira_oikhberg}.
\end{proof}

\begin{proof}[Proof of Theorem \ref{ST_fail}] $(i)\Rightarrow (ii)$.  Assume that $\{A_n\}$ fails Shapiro's theorem in $X$. Then there exists $\{\varepsilon_n\}\searrow   0$ such that, for every $x\in X$ there is a constant $C(x)>0$ such that $E(x,A_n)\leq C(x)\varepsilon_n$ for $n=0,1,\ldots$.  It follows from Lemma \ref{sucesiones} that we may assume, without loss of generality,

$\varepsilon_n\leq 2\varepsilon_{K(n+1)-1}$ for all $n\in\mathbb{N}$.

Obviously $X=\bigcup_n\Gamma_n$, where $\Gamma_{\alpha}=\{x\in X:  E(x,A_n)\leq \alpha \varepsilon_n$ for all $n\in \mathbb{N}\}$.  The sets $\Gamma_n$ are closed subsets of $X$, so that Baire's category theorem implies that  $\Gamma_{m_0}$ contains an open ball $B(x_0,r_0)$ for a certain $m_0\in \mathbb{N}$. Furthermore, it is easy to check that the sets $\Gamma_{\alpha}$ satisfy the symmetry condition $\Gamma_{\alpha}=-\Gamma_{\alpha}$, so that $B(-x_0,r_0)\subseteq \Gamma_{m_0}$. Let $x,y\in \Gamma_{m_0}$. Then
\begin{eqnarray*}
E(\frac{x+y}{2},A_{K(n)}) &\leq &  E(\frac{x}{2},A_{n})+E(\frac{y}{2},A_{n}) \text{ (since } A_n+A_n\subseteq A_{K(n)})\\
&\leq &  E(x,A_{n})+E(y,A_{n})  \text{ (since } d \text{ is non-decreasing)}\\
&\leq & 2m_0 \varepsilon_n  \ (n=0,1,\ldots)
\end{eqnarray*}
Let us now take $j\in\mathbb{N}$ be an (arbitrary) natural number. Then there exists a unique $n\in\mathbb{N}$ such that $K(n)\leq j\leq K(n+1)-1$. Hence
\[
E(\frac{x+y}{2},A_{j}) \leq E(\frac{x+y}{2},A_{K(n)}) \leq 2m_0 \varepsilon_n \leq 2Cm_0 \varepsilon_{K(n+1)-1}\leq  4m_0 \varepsilon_j,
\]
so that $\frac{x+y}{2}\in\Gamma_{4m_0}$.
It follows that
\[
B(0,r_0)\subseteq \frac{1}{2}(B(x_0,r_0)+B(-x_0,r_0)) \subset \Gamma_{4m_0}.
\]
Hence $\{A_n\}$ fails Shapiro's theorem uniformly on $B(0,r_0)$.

$(ii)\Rightarrow (i)$. If  $\{A_n\}$ fails Shapiro's theorem uniformly on $B(0,r_0)$, then there exists $\{\varepsilon_n\}\searrow   0$  such that, $E(x,A_n)\leq  \varepsilon_n$ for all $n\in\mathbb{N}$ and all $x\in B(0,r_0)$. On the other hand,  the balls are absorbing subsets of $X$, so that, for any $x\in X$ there exists $\lambda>0$ such that $x\in\lambda B(0,r_0)$. Then $x=\lambda y$ for some $y\in B(0,r_0)$ and
\[
E(x,A_n)=E(\lambda y,A_n)\leq ([\lambda]+1) E(y,A_n)\leq ([\lambda]+1)  \varepsilon_n, \text{ for all } n\in\mathbb{N}.
\]
\end{proof}

\begin{remark} It easily  follows from the proof of  $(i)\Rightarrow (ii)$ in Theorem \ref{ST_fail} that this implication holds true as soon as the metric $d$ satisfies $d(\frac{x}{2},0)\leq d(x,0)$ for all $x\in X$.  Analogously, the implication  $(ii)\Rightarrow (i)$ holds true as long as $d$ satisfies  $f(\lambda)=\sup_{x\neq 0}\frac{d(\lambda x,0)}{d(x,0)}<\infty$ for all $\lambda\in\mathbb{R}$.
\end{remark}

\begin{remark}  Theorem \ref{ST_fail} generalizes Theorem  \ref{teo_introd} to $F$-spaces since, in the case of (quasi-) Banach spaces we have that $E(rx,A_n)=|r|^pE(x,A_n)$ (with $p=1$ for the Banach setting), so that $E(S(X),A_n)=E(B(0,1),A_n)$ and $E(B(0,r),A_n)=r^pE(B(0,1),A_n)$ for all $n\in\mathbb{N}$ and all $r>0$. In particular, this implies that $\inf_{n\in\mathbb{N}}E(B(0,r),A_n)>0$ for all $r>0$ if and only if $\inf_{n\in\mathbb{N}}E(S(X),A_n)>0$.
\end{remark}

We use Theorem \ref{ST_fail} to prove the following important result:

\begin{proposition}\label{equiv} Let $d_1,d_2$ be two equivalent metrics over the same metric vector space $X$ and let us assume that $\{A_n\}$ is an approximation scheme on $X$. Then the following are equivalent claims:
\begin{itemize}
\item[$(i)$] There exists $r_1>0$ such that  $\{A_n\}$ fails Shapiro's theorem uniformly on the ball $B_{d_1}(0,r_1)$.
\item[$(ii)$] There exists $r_2>0$ such that  $\{A_n\}$ fails Shapiro's theorem uniformly on the ball $B_{d_2}(0,r_2)$.
\end{itemize}
Consequently, if $d_1,d_2$ are both non decreasing equivalent metrics defining an $F$-space $X$, then $\{A_n\}$ satisfies Shapiro's theorem with respect to $d_1$ if and only if $\{A_n\}$ satisfies Shapiro's theorem with respect to $d_2$.
\end{proposition}
\begin{remark} An important result by Klee \cite{klee} (see also \cite[Theorem 1.4.4.]{rolewicz}) guarantees that, if $(X,d)$ is an $F$-space and $d'$ is a metric which is equivalent to $d$ and translation invariant, then $(X,d')$ is also an $F$-space (i.e., $X$ is complete with respect to $d'$). This may be not the case if the metric $d'$ is not translation invariant! Fortunately all metrics in this paper are assumed to be translation invariant.
\end{remark}
\begin{proof}[Proof of Proposition \ref{equiv}]  By definition, $d_1,d_2$ are equivalent metrics in $X$ if there exists functions $\varphi,\phi:(0,\infty)\to (0,\infty)$ such that
\begin{eqnarray*}
&\ & B_{d_2}(0,\varphi(r))\subseteq B_{d_1}(0,r) \\
&\ & B_{d_1}(0,\phi(r))\subseteq B_{d_2}(0,r) \\
\end{eqnarray*}
for all $r>0$. Take $\tau_0=\varphi(1)$ and $0<r<\tau_0$. Then
\begin{eqnarray*}
&\ & B_{d_2}(0,r)\subseteq B_{d_2}(0,\tau_0)\subseteq B_{d_1}(0,1)
\end{eqnarray*}
and there exists
\[
\phi^*(r):=\inf\{s>0: B_{d_2}(0,r)\subseteq B_{d_1}(0,s)\}.
\]
In particular, $B_{d_2}(0,r)\subseteq B_{d_1}(0,\phi^*(r))$. Analogously, for $r<\tau_1=\phi(1)$ the function
\[
\varphi^*(r):=\inf\{s>0: B_{d_1}(0,r)\subseteq B_{d_2}(0,s)\}
\]
is well defined and satisfies $B_{d_1}(0,r)\subseteq B_{d_2}(0,\varphi^*(r))$.


Let us prove that $\lim_{r\to 0} \varphi^*(r)=0$. Obviously, $\varphi^*$ is an increasing function, since $\delta_1<\delta_2$ implies $B_{d_1}(0,\delta_1)\subseteq B_{d_1}(0,\delta_2)$, so that
$\lim_{r\to 0} \varphi^*(r)= \inf_{r> 0} \varphi^*(r)=\rho\geq 0$. Assume that $\rho>0$. Then
\[
B_{d_1}(0,\phi(\frac{\rho}{2}))\subseteq B_{d_2}(0,\frac{\rho}{2}),
\]
and
\[
\rho \leq \varphi^*(\phi(\frac{\rho}{2}))=\inf\{s:  B_{d_1}(0,\phi(\frac{\rho}{2})) \subseteq B_{d_2}(0,s)\}\leq \frac{\rho}{2},
\]
which is impossible. It follows that $\lim_{r\to 0} \varphi^*(r)=0$.
Assume $(i)$. Let $\{\varepsilon_n\}\searrow 0$ and $r_1$ be such that
\[
E_{d_1}(x,A_n)\leq \varepsilon_n \text{ for all } x\in B_{d_1}(0,r_1) \text{ and all } n\in\mathbb{N}.
\]
Let $r_2=\varphi(r_1)>0$ and let $x\in B_{d_2}(0,r_2)\subseteq B_{d_1}(0,r_1)$. For each $n\in\mathbb{N}$ there exists $a_n\in A_n$ such that
\[
d_1(x,a_n)=d_1(x-a_n,0)\leq 2 E_{d_1}(x,A_n)\leq 2\varepsilon_n.
\]
In other words, we have that $x-a_n\in B_{d_1}(0,2\varepsilon_n) \subseteq  B_{d_2}(0,\varphi^*(2\varepsilon_n))$, so that $E_{d_2}(x,A_n)\leq \varphi^*(2\varepsilon_n)$. Now, the sequence $\{\varphi^*(2\varepsilon_n)\}$ is decreasing, converges to zero and does not depend on $x\in B_{d_2}(0,r_2)$. This ends the proof of $(i)\Rightarrow (ii)$. The implication $(ii)\Rightarrow (i)$ follows with the very same arguments.

The last part of this proposition follows as an easy corollary of the first part and Theorem  \ref{ST_fail}.
\end{proof}

Theorem \ref{ST_fail} characterizes approximation schemes satisfying Shapiro's theorem on (a large class of) $F$-spaces. Now, a natural question is if this characterization is useful for studying some concrete examples (otherwise, it would be a nice but inapplicable result). Fortunately, the theorem can be used for some classical cases. In particular, if we perform extra computations, which lead to a generalization of Riesz's lemma for the $F$-spaces setting that was proved by Albinus \cite{albinus}, and we can characterize non-trivial linear approximation schemes satisfying Shapiro's theorem:

\begin{theorem}[Albinus] \label{albinus_th}Let $(X,d)$ be an $F$-space with non decreasing metric $d$ and let $\{A_n\}$ be a non trivial linear approximation scheme on $X$. Then $\{A_n\}$ satisfies Shapiro's theorem on $X$ if and only if $\inf_{n\in\mathbb{N}}E(X,A_n)>0$.
\end{theorem}


\begin{lemma}[Albinus-Riesz Lemma] Let $(X,d)$ be a metric vector space, $M$ a  vector subspace of $X$, and $r>0$. Then
\[
E(B(0,r),M)=\min\{r,E(X,M)\}.
\]
\end{lemma}
\begin{proof} The inequality $E(B(0,r),M)\leq E(X,M)$ is obvious since $B(0,r)\subseteq X$. Moreover, $E(B(0,r),M)\leq r$ because $0\in M$ implies that $E(x,M)\leq d(x,0)$ for all $x$. This proves $E(B(0,r),M)\leq \min\{r,E(X,M)\}$. To prove the other inequality we only need to check, if $E(B(0,r),M)<r$, then $E(B(0,r),M) = E(X,M)$ .

Let us assume that $E(B(0,r),M)<s<r$. If $B(0,s)+M\neq X$ there exists $x\in X$ such that $E(x,M) >s$. Define the function $\varphi(t)=E(tx,M)$. It is easy to prove that $\varphi$ is continuous and $\varphi(0)=0$, $\varphi(1)>s$, so that there exists $\tau\in (0,1)$ such that $\varphi(\tau)=s$. What is more we can take a sequence $\{\tau_n\}\subseteq [0,1]$ such that $\{\varphi(\tau_n)\}$ is increasing and converges to $s$ (this is so because $\phi(0)=0<s$).  Let $n\in\mathbb{N}$ and set $z_n=\tau_nx$. Then $E(z_n,M)<s$, which implies that there exists $m_n\in M$ such that $z_n=m_n+a_n$ with $d(a_n,0)<s$. It follows that  $E(a_n,M)=E(z_n,M)<s$ (since $M$ is a vector space) and $a_n\in B(0,s)$. Hence $E(z_n,M)\leq E(B(0,s),M)\leq E(B(0,r),M)<s<r$ for all $n$. On the other hand,
\[
E(B(0,r),M)<s=\lim_{n\to\infty}E(z_n,M)\leq E(B(0,r),M),
\]
which is impossible. This proves that $B(0,s)+M =X$, so that $E(X,M)\leq s$. Hence, if $E(B(0,r),M)<r$, then $E(B(0,r),M) = E(X,M)$, which is what we wanted to prove.
\end{proof}

\begin{proof}[Proof of Albinus's theorem]  Theorem \ref{ST_fail}  guarantees  that $\{A_n\}$ fails Shapiro's theorem if and only if there exists $r_0>0$ such that $\{E(B(0,r_0),A_n)\}\searrow 0$. On the other hand,  Albinus-Riesz's lemma claims that $E(B(0,r_0),A_{n})=\min\{r_0,E(X,A_{n})\}$  for all $n$. Hence $\{A_n\}$ fails Shapiro's theorem if and only if $\lim_{n\to\infty}E(X,A_n)=0$. In other words,  $\{A_n\}$ satisfies Shapiro's theorem if and only if $\inf_{n\in\mathbb{N}}E(X,A_n)>0$.
\end{proof}

\begin{definition} Let $(X,d)$ be a linear metric space and let $\|x\|=d(x,0)$. Let $V\subseteq X$ be a linear subspace of $X$. We define the radius of $V$ as
$$R_{\|\cdot\|}(V)=\inf_{v\in V\setminus\{0\}}\sup_{t>0}\|tv\|$$
(nothing that, this radius can be infinity).  We say that $(X,d)$ contains short lines if $R_{\|\cdot\|}(X)=0$.
\end{definition}

\begin{proposition} Let us assume that $X$ is an $F$-space with non-decreasing metric and $\{ X_n\}$ is a nontrivial linear approximation scheme on $X$ such that $\dim X_n<\infty$ for all $n$. If
\[
R_{\|\cdot\|}(\bigcup_nX_n)>0,
\]
then $\inf_nE(X,X_n)>0$ and $\{X_n\}$ satisfies  Shapiro's theorem.
\end{proposition}

\begin{proof}

We prove that, if $\inf_nE(X,X_n)=0$, then $R_{\|\cdot\|}(\bigcup_nX_n)=0$. Let us assume that $\varepsilon_n=E(X,X_n)$ satisfies $\lim_{n\to\infty}\varepsilon_n=0$. Given $\varepsilon>0$, we take $n$ such that $\varepsilon_n<\varepsilon$. Let $a\in\bigcup_{k=0}^{\infty}X_k$ be such that $a\not\in X_n$, and let us consider the vector space $Y=X_n+\mathbf{span}\{a\}\subseteq \bigcup_{k=0}^{\infty}X_k$. By hypothesis, $X_n+B(\varepsilon)=X$, since $X_n+B(\varepsilon_n)=X$ and $\varepsilon_n<\varepsilon$. It follows that $Y=X_n+(B(\varepsilon)\cap Y)$, since $X_n$ is a vector subspace of $Y$, so that, if $y=x+h\in Y$ with $x\in X_n$ and
$h\in B(\varepsilon)$, then $h=y-x\in Y$.  Take $m\in\mathbb{N}$ and consider the vector $ma\in Y$. Then $ma=x_m+h_m$ with $x_m\in X_n$ and $h_m\in B(\varepsilon)$.  We can consider, over $Y$, a norm $\|\cdot \|_{\#}$ defining the same topology as $\|\cdot\|$, since $\dim Y<\infty $ (see \cite[Page 16]{rudin}). Furthermore, the sequence $\{h_m\}_{m=0}^{\infty}$ cannot be bounded with respect to $\|\cdot\|_{\#}$, since all norms over $Y$ are equivalent norms and $m\to\infty$ (to prove this assert just take into account that, if $\|\cdot\|_n$ is any norm over $X_n$, then $\|x_n+\lambda a\|^*=\|x_n\|_n+|\lambda|$ defines a norm on $Y$ and $\{h_m\}_{m=0}^{\infty}$ is unbounded with respect to this norm). We may, then, assume that
$\|h_m\|_{\#}>1$ for all $m\geq m_0$ and consider the new sequence $\{\|h_m\|_{\#}^{-1}h_m\}_{m=m_0}^{\infty}$, which is bounded with respect to the norm $\|\cdot\|_{\#}$. This implies that there exists a converging subsequence. Thus, we may assume with no loss of generality that  $\{\|h_m\|_{\#}^{-1}h_m\}_{m=m_0}^{\infty}$ converges. Let
$\omega = \lim_{m\to\infty}\|h_m\|_{\#}^{-1}h_m$. Then $\omega \neq 0$ since it belongs to the unit sphere of $Y$ with respect to the norm $\|\cdot\|_{\#}$, and  $\omega \in B(\varepsilon)\cap Y$, since $\|\cdot\|$ is non-decreasing and $\|h_m\|_{\#}^{-1}<1$ for all $m\geq m_0$. Furthermore,  $\|t\omega\|\leq \varepsilon$ for all $t\in\mathbb{R}$ since, given $t\neq 0$, there exists $m_0(t)\in\mathbb{N}$ such that  $|t|\|h_m\|_{\#}^{-1}<1$ for all $m\geq m_0(t)$, so that
\[
\|t\omega\| = \lim_{m\to\infty}\|t\|h_m\|_{\#}^{-1}h_m\|=  \lim_{m\to\infty}|t|\|h_m\|_{\#}^{-1}\|h_m\|\leq \varepsilon.
\]
This proves that $R_{\|\cdot\|}(\bigcup_nX_n)=0$, since $\varepsilon$ was arbitrary.
\end{proof}

The following result shows a simple sufficient condition for an approximation scheme $\{A_n\}$ to satisfy Shapiro's theorem on $F$-spaces.

\begin{proposition} \label{proponueva} Let $(X,d)$ be an $F$-space with non decreasing metric $d$ and let $\{A_n\}$ be an approximation scheme on $X$. If there exist $\mathbb{N}_0\subseteq \mathbb{N}$ an infinite sequence of natural numbers,
$\{x_n\}_{n\in\mathbb{N}_0}$ a bounded subset of $X$ and $c>0$ such that, for all $n\in\mathbb{N}_0$ we have that
\[
c\leq E(x_n,A_n),
\]
then $\{A_n\}$ satisfies Shapiro's theorem on $X$.
\end{proposition}

\begin{proof}
We proceed by contradiction. If we assume that  $\{A_n\}$ fails Shapiro's theorem on $X$, Theorem \ref{ST_fail} guarantees that  $\{A_n\}$ fails Shapiro's theorem uniformly on a ball $B(0,r)$ for a certain $r>0$, since the metric $d$ is non decreasing. This means that there exist $r>0$ and $\{\varepsilon_n\}\searrow 0$ such that $E(x,A_n)\leq \varepsilon_n$ for all $n\in\mathbb{N}$ and all $x\in B(0,r)$. Let us now assume that $\{x_{n_k}\}_{k\in\mathbb{N}}$ is a bounded subset of $X$,  $c>0$ and we have that
$c\leq E(x_{n_k},A_{n_k}) $
 for all $k\in\mathbb{N}$, and $\lim_{k\to\infty}n_k=+\infty$. Take $\lambda>0$ such that $\{x_{n_k}\}_{k=0}^\infty\subseteq \lambda B(0,r)$ and write $x_{n_k}=\lambda z_k$ with $z_k\in B(0,r)$. Then   $E(z_k,A_n)\leq \varepsilon_n$ for al $n$. It follows  that
 \[
 c\leq E(x_{n_k},A_{n_k}) =E(\lambda z_{k},A_{n_k}) \leq ([\lambda]+1)E( z_{k},A_{n_k}) \leq ([\lambda]+1)\varepsilon_{n_k}  \ \ (k=0,1,\cdots),
 \]
which is impossible, since $\{\varepsilon_n\}\searrow 0$.
\end{proof}

\begin{corollary} Let $(X,d)$ be a locally convex metrizable topological vector space with metric $d(x,y)=\max_{k\in\mathbb{N}}\{2^{-k}\frac{p_k(x-y)}{1+p_k(x-y)}\}$, where $\{p_k\}$ is a separating family of semi norms of $X$. Assume that $(M_k)\subset [0,\infty)$, $\mathbb{N}_0$ is an infinite subset of $\mathbb{N}$ and there exist $\{x_n\}_{n\in\mathbb{N}_0}\subseteq X$, $m_0\in \mathbb{N}$ and $\delta>0$ such that:
\begin{itemize}
\item[$(a)$] $p_k(x_n)\leq M_k$ for all $k\in\mathbb{N}$ and $n\in\mathbb{N}_0$.
\item[$(b)$] $E_{m_0}(x_n,A_n):=\inf_{a\in A_n}p_{m_0}(x_n-a)>\delta$ for all $n\in\mathbb{N}_0$.
\end{itemize}
Then $\{A_n\}$ satisfies Shapiro's theorem on $X$.
\end{corollary}

\begin{proof} Condition $(a)$ guarantees that $\{x_n\}_{n\in\mathbb{N}_0}$ is a bounded subset of $X$. On the other hand,  $\xi(t)=t/(1+t)$ is an increasing function on $(0,\infty)$, which implies that, for $n\in\mathbb{N_0}$,
\begin{eqnarray*}
0<c &=& 2^{-m_0}\frac{\delta}{1+\delta}<2^{-m_0}\frac{E_{m_0}(x_n,A_n)}{1+E_{m_0}(x_n,A_n)} \\
&\leq & \max_{m\in\mathbb{N}}\{2^{-m}\frac{E_{m}(x_n,A_n)}{1+E_{m}(x_n,A_n)}\} \\
& = & \max_{m\in\mathbb{N}}\{\inf_{a\in A_n}2^{-m}\frac{p_{m}(x_n-a)}{1+p_{m}(x_n-a)}\} \\
& \leq & \inf_{a\in A_n}\{ \max_{m\in\mathbb{N}}2^{-m}\frac{p_{m}(x_n-a)}{1+p_{m}(x_n-a)}\} \\
&=& E(x_n,A_n)
\end{eqnarray*}
and we can use Proposition \ref{proponueva}.
\end{proof}

We end this section with the observation that there are infinite dimensional  $F$-spaces with non-decreasing metrics which do not contain approximation schemes satisfying Shapiro's theorem.

\begin{example} Let $\mathbf{s}$ denote the $F$-space of all sequences of real numbers $\{a_n\}_{n=1}^\infty$ with the metric $d(\{a_n\},\{b_n\})=\sum_{k=1}^\infty\frac{1}{2^k}\frac{|a_k-b_k|}{1+|a_k-b_k|}$. Then $d$ is a non-decreasing metric and every approximation scheme $\{A_n\}$ in  $(\mathbf{s},d)$ fails Shapiro's theorem uniformly on $\mathbf{s}$.
\end{example}
\begin{proof} Given $N\in\mathbb{N}$ we denote by $\mathbf{s}_N$ the space of all sequences $\{a_n\}$ such that  $a_{N+k}=0$ for all $k\geq 1$. Let  $\{A_n\}$ be an approximation scheme in  $(\mathbf{s},d)$. It is easy to check that the sets $B_n=\phi_N(A_n)$, $n=0,1,\cdots$ (where $\phi_N(\{a_k\}_{k=1}^\infty)=\{a_{Nk}\}_{k=1}^\infty$, $a_{Nk}=a_k$ if $k\leq N$ and $a_{Nk}=0$ if $k>N$) define an approximation scheme on $\mathbf{s}_N$. Hence Proposition \ref{cero} implies that $\mathbf{s}_N =B_{m(N)+k}$ for all $k\geq 0$ and a certain $m(N)<\infty$.

Let $x\in\mathbf{s}$ and let $a\in A_{m(N)}$ be such that $\phi_N(x)=\phi_N(a)$. Then $d(x,a)\leq \sum_{k=1}^\infty\frac{1}{2^{N+k}}=\frac{1}{2^N}$. Hence $E(\mathbf{s},A_{m(N)})\leq  2^{-N}$. It follows that $\{E(\mathbf{s},A_n)\}\searrow 0$ and $E(x,A_n)\leq E(\mathbf{s},A_n)$ for all $x\in X$ and all $n\in\mathbb{N}$. This ends the proof.

\end{proof}

\section{Shapiro's theorem for  s-numbers and other scales of numbers}

A careful inspection of Theorem \ref{ST_fail}  shows that its proof rests on the construction of the sets $\Gamma_n$, which should be closed and have certain symmetry properties, and the fact that $E(x+y,A_{K(n)})\leq E(x,A_n)+E(y,A_n)$. This suggest that a lethargy result can also be proved in other contexts. Concretely, we introduce the following concept, which admits as particular cases some well known scales of numbers:

\begin{definition} Let $(X,d)$ be a metric vector space and set $\|x\|=d(x,0)$. We say that the map $\mathcal{E}:X\to \ell^\infty$, $\mathcal{E}(x)=\{e_n(x)\}_{n=0}^\infty$ defines an  scale on $X$ if
\begin{itemize}
\item[$(i)$]  $C_1\|x\| \geq e_{n}(x)\geq e_{n+1}(x)$ for all $x\in X$, all $n\in\mathbb{N}$, and a certain constant $C_1>0$. Furthermore, the function $e_n:X\to [0,\infty)$ is continuous for all $n\in\mathbb{N}$.
\item[$(ii)$] There exist a strictly increasing function $K:\mathbb{N}\to\mathbb{N}$ (which we call the ``jump function'') and a constant $C_2>0$ such that $e_{K(n)}(x+y)\leq C_2(e_n(x)+e_n(y))$  for all $x,y \in X$ and $n\in\mathbb{N}$.
\item[$(iii)$] There exists a control function $\phi:[0,\infty)\to [0,\infty)$ such that $e_n(\lambda x)\leq \phi(|\lambda |) e_n(x)$ for all scalar $\lambda$ and all $x\in X$. Furthermore, $e_n(x)=e_n(-x)$ for all $x\in X$.
\end{itemize}
\end{definition}

\begin{example}
Let $\mathcal{L}$ denote the class of linear continuous operators defined between two quasi-Banach spaces. Following Pietsch \cite{Pie,PieID}, a rule $s:\mathcal{L}\to\mathbb{R}^{\mathbb{N}}$ defines an s-number sequence if it satisfies the following properties:
\begin{itemize}
\item[$(i)$] Monotonicity: $\|T\|=s_1(T)\geq s_2(T)\geq \cdots \geq s_n(T)\geq s_{n+1}(T)\geq \cdots \geq 0$.
\item[$(ii)$] Additivity:  $s_{n+m-1}(T+S)\leq s_m(T)+s_n(S)$ for all $T,S\in\mathcal{L}(X,Y)$.
\item[$(iii)$] Ideal property :  $s_{n}(STR)\leq \|S\|s_n(T)\|R\|$ for all $S\in\mathcal{L}(Z,W)$,$T\in\mathcal{L}(Y,Z)$, and $R\in\mathcal{L}(X,Y)$.
\item[$(iv)$] Rank property:  If $\mathbf{rank}(T)<n$, then $s_{n}(T)=0$.
\item[$(v)$] Norming property:  $s_{n}(1_{\ell^2_n})=1$ for all $n\in\mathbb{N}$, where $1_{\ell^2_n}$ denotes the identity operator defined on the n-dimensional Hilbert space $\ell^2_n$.
\end{itemize}
Obviously, if $s$ is an s-number sequence, then $s$ defines an scale on $\mathcal{L}(X,Y)$ for all pair of quasi-Banach spaces $X,Y$.
\end{example}

\begin{example}
Given $X$  a quasi-Banach space and $Q=\{Q_n\}_{n=0}^{\infty}\subseteq \mathcal{P}(X)$
an infinite family of subsets of $\mathcal{P}(X)$, we say that $(X,Q)$ is a {\it generalized approximation scheme} (or that $(Q_n)$ is a generalized approximation scheme in $X$) if:
\begin{itemize}
\item[$(GA1)$] There exists a map $K:\mathbb{N}\to\mathbb{N}$ such that $K(n)\geq n$ and $A_n\in Q_n$ implies $A_n+A_n\subseteq B_{K(n)}$ for certain $B_{K(n)}\in Q_{K(n)}$, for all $n\in\mathbb{N}$.

\item[$(GA2)$] $\lambda A_n\subset A_n$ for all $A_n\in Q_n$, all $n\in\mathbb{N}$ and all scalars $\lambda$.

\item[$(GA3)$] $\bigcup_{n\in\mathbb{N}}\bigcup_{A_n\in Q_n}A_n$ is a dense subset of $X$.
\end{itemize}
Given a generalized approximation scheme $(X,Q)$ and $T\in \mathcal{L}(X)$, we define the approximation numbers associated with $Q$ by
\[
\alpha_n(T,Q)=\inf\{\|T-S\|:S\in\mathcal{L}(X), S(X)\in Q_n\}.
\]
and the Kolmogorov diameters of $T$ associated with $Q$ by
\[
\delta_n(T,Q)=\inf\{r>0: \exists A_n\in Q_n, T(B(0,1))\subseteq rB(0,1)+A_n\}
\]
It is easy to check that the maps $\alpha_Q,\delta_Q:\mathcal{L}(X)\to\ell^{\infty}$, given by $\alpha_Q(T)=\{\alpha_n(T,Q)\}$ and $\delta_Q(T)=\{\delta_n(T,Q)\}$, define scales on $\mathcal{L}(X)$.
\end{example}

\begin{example} Let $(X,d)$ be a compact metric space and let $C(X)$ denote the space of continuous functions $f:X\to \mathbb{R}$ dotted with the norm $\|f\|_{\infty}=\sup_{x\in X}|f(x)|$. Given $f\in C(X)$, the modulus of continuity of $f$ is the function $\omega(f,\delta)=\sup_{d(x,y)\leq \delta}|f(x)-f(y)|$. Then $\Omega(f)=\{\omega(f,\frac{1}{1+n})\}_{n=0}^{\infty}$ defines an scale on $C(X)$.
\end{example}

\begin{proposition} \label{cierre} Let $(X,d)$ be an $F$-space with non-decreasing metric $d$  and let $\mathcal{E}(x)=\{e_n(x)\}_{n=0}^\infty$ be a scale on $X$. Then $X_{\mathcal{E}}:=\{x\in X: \mathcal{E}(x)\in c_0\}$ is a closed vector subspace of $X$.
\end{proposition}

\begin{proof} Let $x,y\in X_{\mathcal{E}}$ and let $\alpha,\beta$ be two scalars. Then $$e_{K(n)}(\alpha x+\beta y) \leq C_2 (e_n(\alpha x)+e_n(\beta y))\leq \max\{\phi(|\alpha|),\phi(|\beta|)\} C_2 (e_n(x)+e_n(y))\to 0 \ \ (n\to\infty).$$
This proves that $X_{\mathcal{E}}$ is a vector subspace of $X$. Let $x\in X$ be such that $x=\lim_{n\to \infty}x_n$ with $\{x_n\}_{n=0}^\infty\subseteq X_{\mathcal{E}}$. Take $\varepsilon>0$ and  $x_k$ such that $C_1C_2\|x-x_k\|<\varepsilon$. Then
$e_{K(n)}(x)\leq C_2(e_n(x-x_k)+e_n(x_k))\leq C_2(C_1\|x-x_k\|+e_n(x_k))\leq \varepsilon + C_2e_n(x_k)\leq 2\varepsilon $ for $n$ big enough. This proves that $x\in X_{\mathcal{E}}$.

\end{proof}

\begin{definition}Let $(X,d)$ be an $F$-space and assume that $\mathcal{E}(x)=\{e_n(x)\}_{n=0}^\infty$ is an  scale on $X$. We say that $\mathcal{E}$ satisfies Shapiro's theorem if for all decreasing sequence $\{\varepsilon_n\}_{n=0}^{\infty}\in c_0$ there exists $x\in X_{\mathcal{E}}$ such that $e_n(x)\neq \mathbf{O}(\varepsilon_n)$.
\end{definition}

\begin{theorem} \label{STscale} Let $(X,d)$ be an $F$-space with non-decreasing metric $d$  and let $\mathcal{E}(x)=\{e_n(x)\}_{n=0}^\infty$ be a scale on $X$. The following are equivalent claims:
\begin{itemize}
\item[$(i)$] There exists $\{\varepsilon_n\}\searrow 0$ such that $e_n(x)=\mathbf{O}(\varepsilon_n)$ for all $x\in X$.
\item[$(ii)$] There exists $\{\varepsilon_n\}\searrow 0$, $C>0$,  and $r_0>0$ such that  $e_n(x)\leq C \varepsilon_n$ for all $x\in B(0,r_0)$.
\end{itemize}
Consequently, the scale $\mathcal{E}$ satisfies Shapiro's theorem on $X$ if and only if  $$\inf_{n\in\mathbb{N}}sup_{x\in B(0,r)\cap X_{\mathcal{E}}}e_n(x)>0 $$ for all $r>0$.
\end{theorem}

\begin{proof} $(i)\Rightarrow (ii)$.  Let  $\{\varepsilon_n\}\searrow   0$ be such that, for every $x\in X$ there is a constant $C(x)>0$ satisfying $e_n(x)\leq C(x)\varepsilon_n$ for $n=0,1,\cdots$.  It follows from Lemma \ref{sucesiones} that we may assume, without loss of generality, that
$\varepsilon_n\leq 2\varepsilon_{K(n+1)-1}$ for all $n\in\mathbb{N}$.  It follows that $X=\bigcup_n\Gamma_n$, where $\Gamma_{\alpha}=\{x\in X:  e_n(x)\leq \alpha \varepsilon_n$ for all $n\in \mathbb{N}\}$.  Now, the sets $\Gamma_n$ are closed subsets of $X$, since the functions $e_n$ are continuous. Hence  Baire's category theorem implies that  $\Gamma_{m_0}$ contains an open ball $B(x_0,r_0)$ for a certain $m_0\in \mathbb{N}$. Finally, the sets $\Gamma_{\alpha}$ satisfy the symmetry condition $\Gamma_{\alpha}=-\Gamma_{\alpha}$, since $e_n(-x)=e_n(x)$.  This implies that $B(-x_0,r_0)\subseteq \Gamma_{m_0}$. Let $x,y\in \Gamma_{m_0}$. Then
\begin{eqnarray*}
e_{K(n)}(\frac{x+y}{2}) &\leq & C_2( e_n(\frac{x}{2})+e_n(\frac{y}{2}) )\\
&\leq &  \phi(\frac{1}{2})C_2 (e_n(x)+e_n(y))  \\
&\leq & 2  \phi(\frac{1}{2})C_2 m_0 \varepsilon_n  \ (n=0,1,\cdots)\\
\end{eqnarray*}
Let us now take $j\in\mathbb{N}$. Then there exists a unique $n\in\mathbb{N}$ such that $K(n)\leq j\leq K(n+1)-1$. Hence
\[
e_j(\frac{x+y}{2}) \leq e_{K(n)}(\frac{x+y}{2}) \leq 2  \phi(\frac{1}{2})C_2 m_0 \varepsilon_n \leq 4  \phi(\frac{1}{2})C_2 m_0 \varepsilon_{K(n+1)-1}\leq  4  \phi(\frac{1}{2})C_2m_0 \varepsilon_j,
\]
so that $\frac{x+y}{2}\in\Gamma_{4  \phi(\frac{1}{2})C_2m_0}$.
It follows that
\[
B(0,r_0)\subseteq \frac{1}{2}(B(x_0,r_0)+B(-x_0,r_0)) \subset \Gamma_{4  \phi(\frac{1}{2})C_2m_0}.
\]
This ends the proof.

$(ii)\Rightarrow (i)$. Assume that $\{\varepsilon_n\}\searrow   0$ satisfies $e_n(x)\leq  \varepsilon_n$ for all $n\in\mathbb{N}$ and all $x\in B(0,r_0)$ and let $x\in X$. Then  there exists $\lambda>0$ such that $x\in\lambda B(0,r_0)$, so that $x=\lambda y$ for some $y\in B(0,r_0)$. This implies that
\[
e_n(x)=e_n(\lambda y)\leq \phi(\lambda) e_n(y)\leq \phi(\lambda)  \varepsilon_n, \text{ for all } n\in\mathbb{N}.
\]
Let us now prove the last claim of the theorem. In  the case $X=X_{\mathcal{E}}$, the result follows easily from $(i)\Leftrightarrow (ii)$. On the other hand, if $X_{\mathcal{E}}$ is a proper subspace of $X$, then Proposition \ref{cierre} guarantees that $X_{\mathcal{E}}$ is an $F$-space when dotted with the metric of $X$. The result follows if we apply the equivalence $(i)\Leftrightarrow (ii)$  to this new space, just taking into account that $B_{X_{\mathcal{E}}}(0,r)=B(0,r)\cap X_{\mathcal{E}}$.
\end{proof}

\begin{corollary} \label{mod} Let $(X,d)$ be a compact metric space which contains infinitely many points and consider the scale on $C(X)$ given by  $\Omega(f)=\{\omega(f,\frac{1}{1+n})\}_{n=0}^{\infty}$. Then $\Omega$ satisfies Shapiro's theorem on $C(X)$.
\end{corollary}

\begin{proof}
It follows from the infinitude of $X$ that for each $\delta>0$ there exists $x,y\in X$, $x\neq y$, $d(x,y)\leq \delta$, since $X$ being compact, it must contain an infinite convergent sequence.
The result  follows from Theorem \ref{STscale} just taking into account that if $x_{n},y_n\in X$ satisfy $0<d(x_n,y_n)\leq \frac{1}{n+1}$, then  $g_n(x)=\frac{2}{\pi}\arctan(\frac{d(x,y_n)}{d(x_n,y_n)})$ belongs to the unit ball of $C(X)$ and $g_n(y_n)=0$, $g_n(x_n)= 1/2$, so that $\omega(g_n,\frac{1}{n+1})\geq \frac{1}{2}$.
\end{proof}

Given a quasi-Banach space $X$ and $E\subseteq X$ a closed subspace, we define $\lambda(E,X)=\inf\{\|P\|:P \text{ is a projection of  } X \text{ onto } E\}$ and
$$p_n(X)=\inf\{\lambda(E,X): E\subset X,\  \ \dim E=n\}.$$

\begin{corollary} \label{coco} Assume that $X$ is a quasi-Banach space such that $\sup p_n(X)=M<\infty$, and $s$ is an s-number sequence. Then $s:\mathcal{L}(X)\to\ell^{\infty}$ satisfies Shapiro's theorem.
\end{corollary}

\begin{proof} Given $n\in\mathbb{N}$, there exist $E_n$ subspace of $X$ with $\dim E_n=n$ and $P_n:X\to X$, projection onto $E_n$ such that $\|P_n\|\leq 2 p_n(X)\leq 2M$. Take $Q_n:X\to E_n$ defined by $Q_n(x)=P_n(x)$ and let $i_n:E_n\to X$ denote the inclusion map. Then $1_{E_n}=Q_nP_ni_n$, so that
\[
1=s_n(1_{E_n})\leq \|Q_n\|s_n(P_n)\|i_n\|\leq Cs_n(P_n).
\]
This implies that $P_n$ satisfies $\|P_n\|\leq M$ and $s_n(P_n)\geq 1/C>0$. The result follows since $s_{n+k}(P_n)=0$ for $k=1,2,\cdots$, since $\mathbf{rank}(P_n)=n$, so that $P_n\in \mathcal{L}(X)_{s}$.
\end{proof}

\begin{remark}
It would be nice to prove a result similar to Corollary \ref{coco} for operators $T:X\to Y$ acting between different Banach spaces, but the result should be complicated since there are examples of s-numbers sequences $s$, nice Banach spaces $X,Y$  and decreasing sequences $\{\varepsilon_n\}\searrow 0$ such that all operator $T\in \mathcal{L}(X,Y)$ satisfies $s_n(T)=\mathbf{O}(\varepsilon_n)$. Concretely, Oikhberg \cite[Proposition 1.2]{Oik} has demonstrated that all operator $T\in\mathcal{L}(c_0,\ell^1)$ satisfies $\lim_{k\to\infty}\sqrt{k}x_k(T)=\lim_{k\to\infty}\sqrt{k}y_k(T)=\lim_{k\to\infty}kh_k(T)=0$, where $x_k(T),y_k(T)$ and $h_k(T)$ denote the Weyl numbers, Chang numbers and Hilbert numbers of the operator $T$, respectively.

On the other hand, Oikhberg \cite[Theorem 1.1]{Oik} has also proved that, for arbitrary infinite dimensional Banach spaces $X,Y$, the sequences of approximation numbers $a_n(T)$ and symmetrized (or absolute) numbers $t_n(T)$ satisfy Shapiro's theorem on $\mathcal{L}(X,Y)$ (for approximation numbers this result was also proved by Almira and Oikhberg in \cite{almira_oikhberg}).

In \cite{AkLe} Aksoy and Lewicki introduced the concept of Bernstein pair with respect to an s-number sequence  $s_n$ as a pair of Banach spaces $(X,Y)$ such that, for all $\{\varepsilon_n\}\searrow 0$ there exists an operator $T\in\mathcal{L}(X,Y)$ and a constant $d>0$ such that $d^{-1}\varepsilon_n\leq s_n(T)\leq d \varepsilon_n$ for all $n\in\mathbb{N}$, and provided some examples of pairs of classical Banach spaces which form a Bernstein pair with respect to the approximation numbers $a_n(T)$ and other s-number scales. In particular, they proved that if $(X,Y)$ is a Bernstein pair with respect to $(s_n)$ and  Suppose there exists a Banach space $W$ which contains an
isometric and complementary copy of $X$ and a Banach space $V$ which contains an isomorphic
copy of $Y$, then $(W,V)$ is a Bernstein pair with respect to $(s_n)$ too \cite[Proposition 3.4]{AkLe}. As a corollary of this result, they proved that $(L_p(0,1),L_q(0,1))$ forms a Bernstein pair with respect to any s-number sequence as soon as $1<p<\infty$ and $1\leq q<\infty$. This result follows from the fact that $(\ell_2,\ell_2)$ is a Bernstein pair with respect to any sequence of s-numbers and,
for $1\leq p<\infty$, $L_p(0,1)$ contains a subspace isomorphic to $\ell_2$ and complemented in $L_p(0,1)$ for $p>1$.

Finally, condition $\sup p_n(X)=M<\infty$ seems to be not superfluous in the  case of Bernstein numbers $b_n(T)$  since Plichko has proved its necessity for this case when $T:X\to H$, $H$ being a Hilbert space  \cite[Proposition 1]{plichko}. What is more, in \cite[Theorem 1]{plichko} the author  proves that if $X$ is a Banach space which contains uniformly complemented $\ell_n^2$'s, then for every Banach space $Y$ the pair $(X, Y )$ is Bernstein with respect to the Bernstein numbers and, in particular, Bernstein numbers satisfy Shapiro's theorem in $\mathcal{L}(X,Y)$.

It is important to note that there are examples of Banach spaces satisfying  $\lim_{n\to\infty} p_n(X) $ $ =\infty$. These examples were constructed by Pisier in 1983 (see \cite{pisier}). Furthermore, if $H$ is a Hilbert space, then $p_n(H) =1$ for all $n$, which gives the opposite behaviour to Pisier's example.
\end{remark}

\section{Slow convergence of sequences of operators}

\begin{definition}  Let $(X,\rho)$ and $(Y,d)$ be two $F$- spaces. Let $T:X\to Y$ be a (possibly nonlinear) operator and $T_n:X\to Y$ be a sequence of (possibly nonlinear) operators. We say that $T_n$ converges almost arbitrarily slowly to $T$ if
\begin{itemize}
\item[$(C1)$] $\lim_{n\to\infty}d(T_nx,Tx)=0$  for all $x\in X$.
\item[$(C2)$] For every $\{\varepsilon_n\}\searrow 0$  there exists $x\in X$ such that $d(T_nx,Tx)\geq \varepsilon_n$ for infinitely many $n\in\mathbb{N}$.
\end{itemize}
We say that $T_n$ converges arbitrarily slowly to $T$ if it satisfies $(i)$ above and
\begin{itemize}
\item[$(C3)$] For every $\{\varepsilon_n\}\searrow 0$  there exists $x\in X$ such that $d(T_nx,Tx) \geq \varepsilon_n$ for all $n\in\mathbb{N}$.
\end{itemize}
If $Z\subset X$, we say that $T_n$ converges almost arbitrarily slowly to $T$ relative to $Z$ if it satisfies $(i)$ above and
\begin{itemize}
\item[$(C4)$] For every $\{\varepsilon_n\}\searrow 0$  there exists $x\in Z$ such that $d(T_nx,Tx)\geq \varepsilon_n$ for infinitely many $n\in\mathbb{N}$.
\end{itemize}
We say that $T_n$ converges arbitrarily slowly to $T$ relative to $Z$ if it satisfies $(i)$ above and
\begin{itemize}
\item[$(C5)$] For every $\{\varepsilon_n\}\searrow 0$  there exists $x\in Z$ such that $d(T_nx,Tx) \geq \varepsilon_n$ for all $n\in\mathbb{N}$.
\end{itemize}
\end{definition}

\begin{remark} \label{notaconvergencia} It is important to note that condition $(C2)$ above is equivalent to
\begin{itemize}
\item[$(C2)'$] For every $\{\varepsilon_n\}\searrow 0$  there exists $x\in X$ such that $d(T_nx,Tx)\neq \mathbf{O}(\varepsilon_n)$.
\end{itemize}
Indeed, assume that $\{\varepsilon_n\}\searrow 0$ and  $x\in X$ satisfies $d(T_nx,Tx)\neq \mathbf{O}(\varepsilon_n)$. Then there exists an strictly increasing sequence of natural numbers $\{n_k\}_{k=1}^\infty$ such that
$d(T_{n_k}x,Tx)\geq k\varepsilon_{n_k}\geq \varepsilon_{n_k}$, $k=1,2,\cdots $.  This proves $(C2)'\Rightarrow (C2)$.  The other implication follows as a consequence of the fact that, for every  sequence  $\{\varepsilon_n\}\searrow 0$ there exists another sequence $\{\epsilon_n\}\searrow 0$ such that  such that  $\lim_{n\to\infty}\frac{\epsilon_n}{\varepsilon_n}=+\infty$ .  Hence, if we assume $(C2)$ and take $x\in X$ such that  $d(T_nx,Tx)\geq \epsilon_n$ for infinitely many $n\in\mathbb{N}$, then  $d(T_nx,Tx)\neq \mathbf{O}(\varepsilon_n)$.

Analogous arguments can be used with condition $(C4)$, which is equivalent to

\begin{itemize}
\item[$(C4)'$] For every $\{\varepsilon_n\}\searrow 0$  there exists $x\in Z$ such that $d(T_nx,Tx)\neq \mathbf{O}(\varepsilon_n)$.
\end{itemize}
\end{remark}

The study of slow convergence of sequences of operators has been  recently studied  by Deutsch and Hundal \cite{DH1,DH2,DH3}. In particular, in \cite{DH1} they introduced the concepts of arbitrarily slowly (respectively,  almost arbitrarily slowly) convergence of a sequence of linear operators $T_n$ to an operator $T$, and  characterized almost arbitrarily slowly convergent sequences as those which are pointwise convergent but not norm convergent.  This result can be generalized to the $F$-spaces setting as follows:
\begin{theorem} \label{nuevo} Let $(X,\rho)$  be an $F$-space  and  $(Y,d)$ be a metric vector space with non-decreasing metric $d$.  Let $T:X\to Y$ be a continuous linear operator and $T_n:X\to Y$ be a sequence of continuous linear operators. The following are equivalent claims:
\begin{itemize}
\item[$(i)$]  The sequence $\{T_n\}$ converges pointwise to $T$ but it does not converge almost arbitrarily slowly to $T$.
\item[$(ii)$] There exists a sequence $\{\varepsilon_n\}\searrow 0$ and a positive real number $r_0>0$ such that
\[
(T_n-T)(B_{\rho}(0,r_0))\subseteq B_d(0,\varepsilon_n), \ \ n=0,1,2,\cdots.
\]
\end{itemize}
In particular, if $(i)$ holds true,  $\{T_n\}$ converges to $T$ in the topology of bounded convergence. Finally, if $X$ is locally bounded, $(i)$ and $(ii)$ are equivalent to
\begin{itemize}
\item[$(iii)$] $\{T_n\}$ converges to $T$ in the topology of bounded convergence.
\end{itemize}
\end{theorem}

\begin{proof} $(i)\Rightarrow (ii)$. Assume $(i)$. It follows from Remark \ref{notaconvergencia} that there exists $\{\varepsilon_n\}\searrow 0$ such that, for each $x\in X$  exists $C(x)>0$ satisfying $d(T_nx,Tx) \leq C(x) \varepsilon_n$ for all $n\in\mathbb{N}$, since $T_n$ does not converge almost arbitrarily slowly to $T$. In particular, $X=\bigcup_{m\in\mathbb{N}}\Delta_m$, where
\[\Delta_m=\{x\in X: d(T_nx,Tx)\leq m\varepsilon_n\text{ for all } n\in\mathbb{N}\}.\]
Obviously, $\Delta_m$ is a closed subset of $X$ since $T,T_n$ are continuous. Furthermore, $\Delta_m=-\Delta_m$ and, if $x,y\in\Delta_m$ then
\begin{eqnarray*}
d(T_n(\frac{x+y}{2}),T(\frac{x+y}{2}) &=& d(\frac{1}{2}(T_n-T)(x+y),0)\leq d((T_n-T)(x+y),0)\\
&\leq&  d((T_n-T)(x),0)+d((T_n-T)(y),0)\leq 2m\varepsilon_n, \ \ (n\in\mathbb{N}),
\end{eqnarray*}
since $d$ is non-decreasing. This implies that
\begin{equation}\label{delta}
\frac{1}{2}(\Delta_m+\Delta_m)\subseteq \Delta_{2m}.
\end{equation}
Baire's category theorem implies that $B_{\rho}(x_0,r_0)\subseteq \Delta_{m_0}$ for certain $m_0\in\mathbb{N}$, $x_0\in X$ and $r_0>0$.  Furthermore,   $B_{\rho}(-x_0,r_0)\subseteq \Delta_{m_0}$, since $\Delta_{m_0}=-\Delta_{m_0}$ and the inclusion (\ref{delta}) shows that
\[
B_{\rho}(0,r_0)\subseteq \frac{1}{2}(B_{\rho}(x_0,r_0)+B_{\rho}(-x_0,r_0))\subseteq \Delta_{2m_0}.
\]
This proves $(i)\Rightarrow (ii)$.

$(ii)\Rightarrow (i)$. Assume $(ii)$ and let $x\in X$. There exists $\lambda >0$ such that $x=\lambda y$ with $y\in B_{\rho}(0.r_0)$, since these balls are absorbing sets. Hence
\begin{eqnarray*}
d(T_nx,Tx)&=&d((T_n-T)x,0)=d(\lambda(T_n-T)y,0)\\
&\leq& ([\lambda]+1)d((T_n-T)y,0)\leq  ([\lambda]+1)\varepsilon_n\  \ (n\in\mathbb{N}),\\
\end{eqnarray*}
which proves $(ii)\Rightarrow (i)$.
The last part of this theorem follows easily from the equivalence $(i)\Leftrightarrow (ii)$ above, in conjunction with the definition of the topology of bounded convergence on the space of linear continuous operators $B(X,Y)$ and the fact that balls are bounded sets in locally bounded metric spaces.
\end{proof}
Theorem \ref{nuevo} has the following nice consequence:

\begin{corollary} Let $(X,\rho)$  be a locally bounded $F$-space  and  $(Y,d)$ be a metric vector space with non decreasing metric $d$.  Let $T:X\to Y$ be a continuous linear operator and $T_n:X\to Y$ be a sequence of continuous linear operators. The following are equivalent claims:
\begin{itemize}
\item[$(i)$]  The sequence $\{T_n\}$ converges almost arbitrarily slowly to $T$.
\item[$(ii)$] $\{T_n\}$ converges pointwise to $T$ but  it does not converge to $T$ in the topology of bounded convergence.
\end{itemize}
\end{corollary}

In \cite{DH1} the authors proved that some classical families of operators are almost arbitrarily slowly convergent to the identity and they stated (without proof) that Bernstein's operators are in fact arbitrarily slowly convergent to the identity. They conjectured that this property should also hold true for other classical operators such as F\'{e}jer's or Landau's. They solved  their conjecture in the positive in \cite{DH3}  by using an appropriate modification of Tjuriemskih's lethargy theorem \cite{tjuriemskih1}. These results have been our main motivation in demonstrating Theorem \ref{conv_oper} below, which transports the ideas of  \cite{DH3} to a much more general context.

 \begin{theorem} \label{conv_oper} Let $(X,\rho)$ and $(Y,d)$ be two $F$- spaces.  Let $\{A_n\}$ be an approximation scheme on $Y$ and let us assume that $\{A_n\}$ satisfies Shapiro's theorem on $Y$. Let $T:X\to Y$ be an operator and and $T_n:X\to Y$ be a sequence of operators such that
\begin{itemize}
\item[$(i)$] $\lim_{n\to\infty}d(T_nx,Tx)=0$  for all $x\in X$.
\item[$(ii)$] $T_n(X)\subseteq A_n$ for all $n\in\mathbb{N}$.
\end{itemize}
Then:
\begin{itemize}
\item[$(a)$] If $T(X)=Y$ then $T_n\to T$ almost arbitrarily slowly. Furthermore, if  $T(X)=Y$ and $Y$ is a Banach space, then $T_n\to T$ arbitrarily slowly.
\item[$(b)$] If $Z$ is a subspace of $Y$ such that $\{A_n\}$ satisfies Shapiro's theorem on $Z$ and $Z\subseteq T(X)$, then  $T_n\to T$ almost arbitrarily slowly. Moreover, if we add the hypothesis that $X=Y$ and $T=I$, then  $T_n\to I$ almost arbitrarily slowly relative to $Z$.
\item[$(c)$] Assume that  $X$, $Y$ are quasi-Banach,  $A_n$ is a vector space  and $\dim A_n<\infty$ for all $n\in\mathbb{N}$, and there exists $Z$, a closed subspace of $Y$, such that  $Z\subseteq T(X)$ and $\dim Z=\infty$. Then  $T_n\to T$ arbitrarily slowly. Moreover, if we add the hypothesis that $X=Y$ and $T=I$, then  $T_n\to I$ arbitrarily slowly relative to $Z$.
\end{itemize}
\end{theorem}
\begin{proof}  $(a)$. Let $\{\varepsilon_n\}$ be a non-increasing sequence converging to $0$. We know that  $E(y,A_n)\neq \mathbf{O}(\varepsilon_n)$ for a certain $y\in Y$, since $\{A_n\}$ satisfies Shapiro's theorem on $Y$. Now, $T(X)=Y$ implies that $y=Tx$ for a certain $x\in X$. This leads to $$d(T_nx,Tx)\neq \mathbf{O}(\varepsilon_n)$$ (which is what we wanted to prove), since $T(X)\subseteq A_n$ implies that  $d(T_nx,Tx)\geq E(Tx,A_n)=E(y,A_n)$. The second part of $(a)$ follows directly from part $(f)$ in Theorem \ref{teo_introd} (i.e., Corollary 3.7 in \cite{almira_oikhberg}).

$(b)$ The arguments are the same as in $(a)$, but now we use that $\{A_n\}$ satisfies Shapiro's theorem on $Z$ and $Z\subseteq T(X)$. Thus, the element $y$ can be chosen from $Z$ and it is still of the form $y=Tx$ for a certain $x\in X$. This leads to almost arbitrarily slowly convergence of $T_n$ to the operator $T$.

$(c)$ Use Theorem \ref{teo_introd4} (i.e., Theorem 4.3 from \cite{almira_oikhberg_2}).
\end{proof}

Theorem \ref{conv_oper}  includes some interesting cases not considered in \cite{DH3}. We include a few of them here for the sake of completeness:
\begin{itemize}
\item Greedy approximation with respect to a complete minimal system in the Banach setting produces sequences of nonlinear operators $T_n$ which are arbitrarily slowly convergent to the identity operator (see \cite[Theorem 6.2]{almira_oikhberg}).
\item Greedy approximation with respect to an unconditional basis $(\phi_n)_{n=0}^{\infty}$ of a separable Banach space $X$ produces sequences of nonlinear operators $T_n$ which are almost arbitrarily slowly convergent to the identity operator relative to any infinite dimensional closed subspace $Z$ of $X$ (use Theorem \ref{teo_introd5} (i.e.,  \cite[Theorem 7.7]{almira_oikhberg_2})).
\item If $T_n:X\to Y$ is a sequence of linear operators, $T_n\to T$ pointwise, $T(X)$ contains a finite codimensional subspace of $Y$, and $\overline{T_n(X)}\neq Y$ for all $n$, then $T_n\to T$ almost arbitrarily slowly (use Theorem \ref{teo_introd2} (i.e, Theorem 2.9 from \cite{almira_oikhberg_2}) and part $(b)$ of Theorem \ref{conv_oper}).
\item Consider the examples given in \cite{DH3} (Bernstein's, Fejer's, etc.). In all these cases we can prove that the sequence of operators is arbitrarily slowly convergent to the identity operator relative to  any infinite dimensional closed subspace $Z$ of $C[a,b]$.  (To prove this, just take into account part $(c)$ of Theorem \ref{conv_oper}).
\end{itemize}

The operators $T$ to which Theorem  \ref{conv_oper} is applicable have large range. In fact, they satisfy $T(X)=Y$ (the whole image space) or $T(X)$ contains a finite-codimensional subspace of $Y$, or it contains an infinite dimensional closed subspace of $Y$. (This holds true if $\dim T(X)=\infty$ and $T$ has closed range. These operators are well known in functional analysis). This is obviously a serious restriction on our theory. For example, if the inclusion of the infinite-dimensional closed subspace $Z\subset T(X)$ is continuous, then $T$ will be not compact. An interesting open question is to search for some kind of description of the class of compact operators $T$ such that $T_n$ converges almost arbitrarily slowly to  $T$ if and only if it converges  arbitrarily slowly to $T$. What is more, perhaps this question makes sense for  strictly singular, finitely strictly singular, or other important classes of operators.

\bigskip

\footnotesize{A. G. Aksoy

Department of Mathematics. Claremont McKenna College.

Claremont, CA, 91711, USA.

email:  aaksoy@cmc.edu}

\bigskip

\footnotesize{J. M. Almira

Departamento de Matem\'{a}ticas. Universidad de Ja\'{e}n.

E.P.S. Linares,  C/Alfonso X el Sabio, 28

23700 Linares (Ja\'{e}n) Spain

email: jmalmira@ujaen.es}

\bigskip


\begin{thebibliography}{13}



\bibitem{Ak0} \textbf{A. Aksoy,}   \textit{ Approximation Schemes, Related s-Numbers and Applications, } Ph. D. Thesis, University of Michigan, 1984.

\bibitem{Ak1} \textbf{A. Aksoy,}  Intermediate Spaces, Turkish Journal of Mathematics, \textbf{13} (3) (1989) 79-90. 

\bibitem{Ak2} \textbf{A. Aksoy,}  $\mathbf{Q}$-compact Sets and $\mathbf{Q}$-compact Maps, Mathematica Japonica, \textbf{36} (1) (1991), 1-7. 

\bibitem{Ak3} \textbf{A. Aksoy, } A Generalization of $n$-widths, in \textit{Approximation Theory, Spline Functions and Applications, } Kluwer Academic Publishers, NATO-ASI Series,  1992,  269-278.

\bibitem{AkLe} \textbf{A. G. Aksoy and G. Lewicki,} Diagonal Operators, s-Numbers and Bernstein Pairs, Note Mat. \textbf{17} (1999) 209-216.

\bibitem{AkNa} \textbf{A. Aksoy and M. Nakamura,}  The Approximation Numbers $a_n(T)$ and $\mathbf{Q}$-precompactness, Mathematica Japonica, \textbf{31} (6) (1986) 827-840.

\bibitem{albinus} \textbf{G. Albinus, } Eine bemerkung zur approximationstheorie in metrisierbaren topologischen vectorr\"{a}umen, Rev. Roum. Math. Pures et Appl. \textbf{14} (10) (1969) 1383-1894.

\bibitem{almirabjma} \textbf{J. M. Almira, } On strict inclusion relations between approximation and interpolation spaces, Banach J. Math. Anal. \textbf{5} (2) (2011) 93-105


\bibitem{almiraluther2} \textbf{J.M. Almira and U. Luther, }  Generalized approximation spaces and applications, Math. Nachr. \textbf{263-264} (2004).

\bibitem{almiraluther1} \textbf{J.M. Almira and U. Luther, } Compactness and generalized approximation spaces, Numer. Funct. Anal. Optim. \textbf{23} (2002) 1--38.

\bibitem{almira_oikhberg} \textbf{J. M. Almira, T. Oikhberg, } Approximation schemes satisfying Shapiro's theorem, J. Approx. Theory \textbf{164} (2012) 534-571.

\bibitem{almira_oikhberg_2} \textbf{J. M. Almira, T. Oikhberg, } Shapiro's theorem for subspaces, J. Math. Anal. Appl. \textbf{388} (2012) 282-302.

\bibitem{aokiresult} \textbf{T. Aoki, } Locally bounded linear topological spaces, Proc. Imp. Acad. Tokyo \textbf{18} (10) (1942) 588-594.

\bibitem{bernsteininverso}  \textbf{S. N. Bernstein, }Sur le probleme
inverse de la th\'{e}orie de la meilleure approximation des functions continues.
Comtes Rendus, \textbf{206} (1938) 1520-1523.(See also: Ob obratnoi zadache
teorii nailuchshego priblizheniya nepreryvnykh funksii, Sochineniya Vol II (1938)

\bibitem{borodin} \textbf{P. A. Borodin, } On the existence of an element with given deviations from an expanding system of subspaces, Mathematical Notes, 80 (5) (2006) 621-630   (Translated from Matematicheskie Zameti 80 (5) (2006) 657-667).


\bibitem{brukru} \textbf{Y. Brudnyi, N. Kruglyak, } On a family of approximation spaces,  In:
\textit{Investigations in function theory of several real variables} Yaroslavl' State Univ., Yaroslavl' (1978),15-42.


\bibitem{butzer_scherer} \textbf{P.L. Butzer and K. Scherer, } \textit{Approximationsprozesse und Interpolationsmethoden}, Bibliographisches Institut Mannheim, Mannheim, 1968.

\bibitem{cobos} \textbf{F. Cobos, }  On the Lorentz-Marcinkiewicz operator ideals, Math. Nachr.  \textbf{126} (1986) 281-300.

\bibitem{cobos_milman} \textbf{F. Cobos and M. Milman, } On a limit class of approximation spaces, Numer. Funct. Anal. Optim. \textbf{11} 1-2 (1990) 11-31

\bibitem{cobos_resina} \textbf{F. Cobos and I. Resina,} Representation theorems for some operator ideals, J. London Math. Soc. \textbf{39} (1989) 324-334.


\bibitem{feher_g} \textbf{F. Feher and G. Gr\"{a}ssler, } On an extremal scale of approximation spaces,
J. Comp. Anal. Appl. \textbf{3} (2) (2001), 95-108.

\bibitem{DH1} \textbf{F. Deutsch, H. Hundal, } Slow convergence of sequences of linear operators I: Almost arbitrarily slow convergence, J. Approx. Theory \textbf{162} (2010) 1701-1716.

\bibitem{DH2} \textbf{F. Deutsch, H. Hundal, } Slow convergence of sequences of linear operators II: Arbitrarily slow convergence, J. Approx. Theory \textbf{162} (2010) 1717-1738.


\bibitem{DH3} \textbf{F. Deutsch, H. Hundal, } A Generalization of Tyuriemskih's Lethargy Theorem and Some Applications,  Numer. Functional Anal. and Optimization, \textbf{34} (9) (2013) 1033-1040.


\bibitem{devorenonlinear}  \textbf{R. A. DeVore,  } Nonlinear approximation, Acta Numer. \textbf{7} (1998) 51-150.

\bibitem{devore} \textbf{R. A. DeVore, G. G. Lorentz, } \textit{Constructive approximation, } Springer, 1993.







\bibitem{HMR1} \textbf{C.V. Hutton, J.S. Morrell, and J.R. Retherford,}  Approximation Numbers and Kolmogoroff Diameters of Bounded Linear Operators,  Bull. Amer. Math. Soc. \textbf{80} (1974) 462-466.
\bibitem{HMR2} \textbf{C.V. Hutton, J.S. Morrell, and J.R. Retherford, } Diagonal Operators, Approximation Numbers and Kolmogoroff Diameters, J. Approx. Theory \textbf{16} (1976) 48-80.

\bibitem{kakutani} \textbf{S. Kakutani, } \"{U}ber die Metrisation der topologischen Gruppen, Proc. Imp. Acad. Tokyo, \textbf{12} (1936) 82-84.

\bibitem{kalton} \textbf{N. J. Kalton, N. T. Peck, J. W. Roberts, }
\textit{An $F$-space sampler, } London Math. Soc. Lecture Note Series \textbf{89},
Cambridge University Press, 1984.

\bibitem{klee} \textbf{V. Klee, } Invariant metrics in groups, Proc. Amer. Math. Soc. Math Scand. \textbf{3} (1952) 484-487.

\bibitem{kolmorogov} \textbf{A. N. Kolmogorov,  } Zur Normierbarkeit eines allgemeinen topologischen Raumes, Studio Matt. \textbf{5} (1935) 29-33.







\bibitem{Oik} \textbf{T. Oikhberg, } Rate of Decay of s-Numbers,  J. Approx. Theory \textbf{163} (2011) 311-327.


\bibitem{peetre_sparr} \textbf{J. Peetre and G. Sparr,}  Interpolation of normed abelian groups, Annali di Matematica Pura ed Applicata \textbf{12} (1972) 216-262.


\bibitem{Pie} \textbf{A.~Pietsch, } Approximation spaces,
Journal of Approximation Theory \textbf{32} (1981) 115--134.
\bibitem{PieID} \textbf{A.~Pietsch,}
\textit{Operator ideals},
North-Holland, Amsterdam, 1980.

\bibitem{Pie2008} \textbf{A. Pietsch,} Bad Properties of the Bernstein Numbers,  Studia Math. \textbf{184} (2008) 263-269.

\bibitem{plichko} \textbf{A. Plichko, } Rate of Decay of the Bernstein Numbers, Journal of Mathematical Physics, Analysis, Geometry. \textbf{9} (1) (2013)  59-72.

\bibitem{pisier} \textbf{G. Pisier, } Counterexamples to a Conjecture of Grothendieck, Acta Math. \textbf{151} (1983) 181-208.


\bibitem{pustylnik} \textbf{E. Pustylnik, } Ultrasymmetric sequence spaces in approximation theory, Collectanea Mathematica, \textbf{57} (3)  (2006) 257-277.

\bibitem{pustylnik2} \textbf{E. Pustylnik,}  A new class of approximation spaces,  Rend. Circ. Mat.
Palermo, Ser. II, Suppl., v. 76 (2005), 517-532.


\bibitem{rolewicz} \textbf{S. Rolewicz, } \textit{Metric linear spaces, 2th Ed.} Mathematics and its applications, East European Series, Kluwer Acad. Publ., 1985.

\bibitem{rolewiczresult} \textbf{S. Rolewicz, } On a certain class of linear metric spaces, Bull. Acad. Pol. Sci. \textbf{5} (1957) 471-473.

\bibitem{rudin} \textbf{W. Rudin, } \textit{Functonal Analysis (2¼ Edition),} Mc-Graw Hill, Inc., 1991.

\bibitem{shapiro}  \textbf{H. S. Shapiro, }Some negative theorems of
Approximation Theory, Michigan Math. J. \textbf{11} (1964) 211--217.



\bibitem{singerlibro} \textbf{I. Singer, } \textit{Best approximation in normed linear
spaces by elements of linear subspaces}, Springer Verlag, New York, 1970.




\bibitem{tita_equiv_quasinorm}  \textbf{N. Ti\c{t}a,}  Equivalent quasi-norms on some operator ideals, Annal. Univ. Craiova \textbf{28} (2001) 16-23.

\bibitem{tita_cluj_99} \textbf{N. Ti\c{t}a,} On a limit class of Lorentz-Zygmund ideals, Analysis, Functional Equations, Approximation and Convexity, 302--306, Cluj Napoca, 1999.


\bibitem{tita_anal} \textbf{N. Ti\c{t}a, } A generalization of the limit class of approximation spaces, Annal. Univ. Iasi \textbf{43} (1997) 133-138.

\bibitem{tita_studia} \textbf{N. Ti\c{t}a,}  Approximation spaces and bilinear operators, Studia Univ. Babes-Bolyai, Ser. Math. \textbf{35} (4) (1990) 89-92.

\bibitem{tita_col} \textbf{N. Ti\c{t}a,} On a class of $\ell_{\Phi,\phi}$ operators,   Collectanea Mathematica, 32 (3)  (1981) 275-279.

\bibitem{tita_col2} \textbf{N. Ti\c{t}a,} $L_{\Phi,\phi}$ operators and $(\Phi,\phi)$ spaces,   Collectanea Mathematica, 30 (1)  (1979) 3-10.

\bibitem{tjuriemskih1} \textbf{I. S. Tjuriemskih, } On a problem of S. N. Bernstein, Uchen. Zap. Kalinin.
Gos. Ped. Inst. \textbf{52} (1967) 123--129 (Russian).



\end{thebibliography}
\end{document}